\documentclass[12pt,a4paper]{article}
\usepackage[top=1in, bottom=1in, left=1in, right=1in]{geometry}
\usepackage[utf8]{inputenc}
\usepackage{amsmath}
\usepackage{textcomp}
\usepackage{amssymb}
\usepackage{diagbox}
\usepackage{ytableau}

\usepackage{youngtab}
\usepackage{graphicx,epstopdf}
\usepackage{ mathrsfs }
\usepackage{graphics}
\usepackage{stmaryrd}
\usepackage{amsfonts}
\usepackage{hyperref}
\usepackage{adjustbox}
\usepackage{amsthm}
\usepackage[skip=2pt]{caption}
\usepackage[figurename=Fig.]{caption}
 \usepackage[usenames,dvipsnames]{pstricks}
 \usepackage{epsfig}
 \usepackage{pst-grad} % For gradients
 \usepackage{pst-plot} % For axes
 \usepackage{algorithm}
 \usepackage[noblocks]{authblk}
\usepackage{wasysym}
\usepackage[utf8]{inputenc}

\newtheorem{thm}{Theorem}[section]
\theoremstyle{definition}

\newtheorem{exmp}{Example}[section]

\begin{document}
\title{\textbf{Total Chromatic Number for Some Classes of Cayley Graphs}}
\author{S. Prajnanaswaroopa, \ J. Geetha, \ K. Somasundaram}

\date{Department of Mathematics, Amrita School of Engineering-Coimbatore\\ Amrita Vishwa Vidyapeetham, India.\\\{s\_prajnanaswaroopa, j\_geetha, s\_ sundaram\}@cb.amrita.edu}

\maketitle

\noindent\textbf{Abstract:} The Total coloring conjecture states that any simple graph $G$ with maximum degree $\Delta$ can be totally colored with at most $\Delta+2$ colors. In this paper, we have obtained the total chromatic number for some classes of Cayley graphs. \\

\noindent \textbf{Keywords:} Cayley graphs; Total coloring; Circulant graphs; Perfect Cayley graphs.
\maketitle
\section{Introduction}

Total coloring of a simple graph is a coloring of the vertices and the edges such that any adjacent vertices or edges or edges and their incident vertices do not receive the same color. The \textit{total chromatic number}  of a graph \textit{G}, denoted by $ \chi''(G) $, is the minimum number of colors that suffice in a total coloring. It is clear that $ \chi''(G) \geq \Delta+1 $, where $\Delta$ is the maximum degree of \textit{G}.  \\

The Total coloring conjecture (TCC) (Behzad \cite{BEZ} and Vizing \cite{VGV} independently proposed) is a famous open problem in graph theory  that states that any simple graph $G$ with maximum degree $\Delta$ can be totally colored with at most $\Delta+2$ colors. The progress on this conjecture had been low in the initial years but several advancements are in progress\cite{GNS}. The graphs that can be totally colored in $\Delta+1$ colors are said to be type I graphs and the graph with the total chromatic number $\Delta+2$ is said to be type II. The other notations used are standard in graph most graph theory literature.\\

A Cayley graph for a group $H$ with symmetric generating set $S\subset H$  is a simple graph with vertices as all the elements of the group; and edges between each two elements of the form $g$ and $gs$, where $g\in H, s\in S$. Note that a symmetric subset implies $s\in S \implies s^{-1}\in S$.The set $S$ is also supposed to not have the identity element of the group. The Circulant graph is a Cayley graph on the group $H=\mathbb{Z}_n$, the cyclic group on $n$ vertices.  The Unitary Cayley graph $U_n$ is defined as the Cayley graph of the group $\mathbb{Z}_n$ with generating set $S=\{i|i<n,\ \ gcd(i,n)=1\}$. Thus, the degree of $U_n$ is $\phi(n)$. Thus, it is also a Circulant graph.\\

%\section{Total Chromatic number of some Unitary Cayley graphs}
\section{Total chromatic number of some classes of Cayley graphs}
It has been proved by J Geetha et al.\cite{GNS} that the Unitary Cayley graphs satisfy the Total Coloring conjecture, that is, their total chromatic number is at most $\Delta+2$, where $\Delta$ be the degree of the graph. The Unitary Cayley graphs $U_n$ for $n$ even is seen to have a better structure. In fact, it is bipartite with the two parts consisting of odd and even numbers less than $n$ respectively. In this case it is proved here  that almost all such graphs have total chromatic number $\Delta +1$, where $\Delta=\phi(n)$ is the degree of the graph.
\begin{thm}
If $n=2^k$, where $k$ is a positive integer, then the total chromatic number of the graph $U_n$, $\chi''(U_n)=\phi(2^k)+2=2^{k-1}+2$
\end{thm}
\begin{proof}
It is easily seen that $\phi(2^k)=2^{k-1}=\frac{2^k}{2}$. Thus, the graph $U_n$ in this case is a complete regular bipartite graph, whence its total chromatic number is $\chi''(U_n)=\Delta+2=\phi(n)+2=2^{k-1}+2$.
\end{proof}
\begin{thm}
If $n=2^km$ or $p^k$, where $m,k$ are positive integers, $p$ is an odd prime and $m>1$ is odd, then the graph $U_n$ is total colorable with $\Delta+1=\phi(n)+1$ colors.
\end{thm}
\begin{proof}
Case 1: $n=2^km.$

Let the least prime that divides $m$ be $r$. Now, the $n$ number of vertices can be divided into $r$ independent sets or $r$ color classes  with $\frac{n}{r}$ number of vertices in each set. The sets can be written as $\{0,r,2r,\ldots,n-r\},\{1,r+1,\ldots,n-r+1\},\ldots, \{r-1,2r-1,\ldots,n-1\}$. This is seen to be a conformable $r$ vertex coloring as the parity of each independent set is even. \\

Now, consider the total color matrix of the graph. We split the total color matrix into two parts- the first part concerns the total coloring of the subgraph which consists of the vertices and $\frac{r-1}{2}$-hamiltonian cycles. This corresponds to labelling of the diagonal entries and $r-1$ superdiagonals (subdiagonals) of the total color matrix.  Since $n-s$ and $s$  generate the same cycle in the graph,  the entry started at the $j-$th superdiagonal continues to the $n+2-j$-th superdiagonal, where $s$ is any element of the generating set of $G$. This forces that the differences between the entries of the $j$-th and $n+2-j$-th superdiagonal should be $(j-1)\mod r$.\\
Let us consider the pairs of numbers $(2+t,r-t)\pmod r,\,\,t\le\lfloor\frac{r-1}{4}\rfloor$ and $(\frac{r+1}{2}+t+1,\frac{r+1}{2}-t)\pmod r,\,\, t\le\lfloor\frac{r+1}{4}\rfloor-1$. Note that whenever the number is 0, it is considered as $r$.%{ We observe that the differences between the numbers in each pair is even in the first case, and odd in the second case, respectively.}

The conditions put on $t$ ensure that the numbers inside the brackets are all distinct and all the differences between the number pairs from $0$ to $\lfloor\frac{r}{2}\rfloor$ are covered, thereby ensuring that all the numbers from $1$ to $r$ are covered \cite{ALE}. \\  

Using this rule, the $r$ entries (the first $\frac{r-1}{2}+1$ entries and the last $\frac{r-1}{2}$ entries) of the first row is chosen by the following rule: (here $j$ is the column number of the color matrix, $j\le\frac{n}{2}$\\
i) If $j=1$ then the entry in that column is $1$.\\
ii) If $j$ is odd, then the entry at the $j$-th column is $(2+\frac{j-3}{2}) \mod r$. Note that this would force the entry  at the $n+2-j$-th column to be  $(r-\frac{j-3}{2})\mod r$.\\
iii)  If $j$ is even, then the entry at the $j$-th column is $(\frac{r+1}{2}+\frac{j-2}{2}+1)\mod r$. Note that this would force the entry  at the $n+2-j$-th column to be $(\frac{r+1}{2}-\frac{j-2}{2})\mod r$.\\

It is seen that the entries generated by the above rules give us distinct entries by the prior discussion. The first column along with the $\frac{r-1}{2}$ even numbers less than $r$, (corresponding to the odd numbers in $\mathbb{Z}_n$ less than $r$ in the generating set of $G$) are the required column numbers which are the $\frac{r-1}{2}$ starting entries of the superdiagonals.    We  start the diagonal in the pattern $1-2-3-\ldots-r$. The $r-1$ entries of the   superdiagonals start from the first entries determined by the above rules and then follow a similar pattern as in the principal diagonal. Let $a_j$ be the entry at the $j^{th}$ column. The entries of the superdiagonals (corresponding to the colors of the edges) clash with the vertex colors only when $a_j\equiv j\mod r$ or $a_{n+2-j}\equiv (n+2-j)\mod r$. Now, using the above rules, these amount to either $2+\frac{j-3}{2}\equiv j \mod r$ and $(r-\frac{j-3}{2})\equiv (n+2-j) \mod r$, when $j$ is odd; or $(\frac{r+1}{2}+\frac{j-2}{2}+1)\equiv j\pmod r$ and $(\frac{r+1}{2}-\frac{j-2}{2})\equiv(n+2-j)\mod r$. Solving these congruences using the fact that $r|n$, we obtain the only case $j\equiv1\pmod r$, which implies $r|(j-1)$, which cannot happen, because $j<r$.  Hence, the colors assigned to the vertices and its incident edges do not clash. In addition, the colors assigned to the edges do not clash because the initial column numbers $j$ are all less than $r$.  By symmetry,  the subdiagonals are also filled. This fills diagonal elements with the colors corresponding to a conformable $r$ vertex coloring. Also,  the entries corresponding to $\frac{r-1}{2}$   hamiltonian cycles. The colors used up to now is $r$.\\

Now, consider the remaining part of the total color matrix. It consists of edge coloring of $\frac{\phi(n)-(r-1)}{2}$ remaining hamiltonian cycles, which are all even, as the graph is bipartite. Since the bipartite graphs are class I, the remaining edges are colored with $\phi(n)-r+1$ extra colors. Thus, total number of colors used  is $\phi(n)-r+1+r=\phi(n)+1$ colors.\\

Case 2:\\

When $n=p^k$, the resulting unitary Cayley graph is a balanced complete multipartite graph with $p$ parts. Hence, by Bermond's theorem \cite{BER}, it is evidently type I.
\end{proof}
\begin{exmp}
The adjacency matrix of the graph $U_{24}$ is shown in the table \ref{Table a}. \\

\begin{table}[h]
\begin{adjustbox}{width=\columnwidth,center}
\begin{tabular}{|c|c|c|c|c|c|c|c|c|c|c|c|c|c|c|c|c|c|c|c|c|c|c|c|c|}\hline&0&1&2&3&4&5&6&7&8&9&10&11&12&13&14&15&16&17&18&19&20&21&22&23\\\hline
0&0&1&&&&1&&1&&&&1&&1&&&&1&&1&&&&1\\\hline1&1&0&1&&&&1&&1&&&&1&&1&&&&1&&1&&&\\\hline2&&1&0&1&&&&1&&1&&&&1&&1&&&&1&&1&&\\\hline3&&&1&0&1&&&&1&&1&&&&1&&1&&&&1&&1&\\\hline4&&&&1 &0 &1 &&&&1&&1&&&&1&&1&&&&1&&1 \\\hline5&1&&&&1 &0 &1 &&&&1&&1&&&&1&&1&&&&1&\\\hline6&&1&&&&1 &0 &1 &&&&1&&1&&&&1&&1&&&&1\\\hline7&1&&1&&&&1 &0 &1 &&&&1&&1&&&&1&&1&&&\\\hline8&&1&&1&&&&1 &0 &1 &&&&1&&1&&&&1&&1&&\\\hline9&&&1&&1&&&&1 &0 &1 &&&&1&&1&&&&1&&1&\\\hline10&&&&1&&1&&&&1 &0 &1 &&&&1&&1&&&&1&&1\\\hline11&1&&&&1&&1&&&& 1&0 &1 &&&&1&&1&&&&1&\\\hline12&&1&&&&1&&1&&&&1 &0 &1 &&&&1&&1&&&&1\\\hline13&1&&1&&&&1&&1&&&&1 &0 &1 &&&&1&&1&&&\\\hline14&&1&&1&&&&1&&1&&&&1 &0 &1 &&&&1&&1&&\\\hline15&&&1&&1&&&&1&&1&&&&1 &0 &1 &&&&1&&1&\\\hline16&&&&1&&1&&&&1&&1&&&&1 &0 &1 &&&&1&&1\\\hline17&1&&&&1&&1&&&&1&&1&&&&1 &0 &1 &&&&1&\\\hline18&&1&&&&1&&1&&&&1&&1&&&&1 &0 &1 &&&&1\\\hline19&1&&1&&&&1&&1&&&&1&&1&&&&1 &0 &1 &&&\\\hline20&&1&&1&&&&1&&1&&&&1&&1&&&&1 &0 &1 &&\\\hline21&&&1&&1&&&&1&&1&&&&1&&1&&&&1 &0 &1 &\\\hline22&&&&1&&1&&&&1&&1&&&&1&&1&&&&1 &0& 1\\\hline23&1&&&&1&&1&&&&1&&1&&&&1&&1&&&&1&0\\\hline\end{tabular}\end{adjustbox}\caption{Adjacency matrix of $U_{24}$}\label{Table a}\end{table}

\vspace{0.3cm}
\newpage
The first part of the total color matrix is as given in the table \ref{Table b}.\\

\begin{table}[h]
\begin{adjustbox}{width=\columnwidth,center}
\begin{tabular}{|c|c|c|c|c|c|c|c|c|c|c|c|c|c|c|c|c|c|c|c|c|c|c|c|c|}\hline&0&1&2&3&4&5&6&7&8&9&10&11&12&13&14&15&16&17&18&19&20&21&22&23\\\hline
0&1&3&&&&&&&&&&&&&&&&&&&&&&2\\\hline1&3&2&1&&&&&&&&&&&&&&&&&&&&&\\\hline2&&1&3&2&&&&&&&&&&&&&&&&&&&&\\\hline3&&&2&1&3&&&&&&&&&&&&&&&&&&&\\\hline4&&&&3 &2 &1 &&&&&&&&&&&&&&&&&& \\\hline5&&&&&1 &3 &2 &&&&&&&&&&&&&&&&&\\\hline6&&&&&&2 &1 &3 &&&&&&&&&&&&&&&&\\\hline7&&&&&&&3 &2 &1 &&&&&&&&&&&&&&&\\\hline8&&&&&&&&1 &3 &2 &&&&&&&&&&&&&&\\\hline9&&&&&&&&&2 &1 &3 &&&&&&&&&&&&&\\\hline10&&&&&&&&&&3 &2 &1 &&&&&&&&&&&&\\\hline11&&&&&&&&&&& 1&3 &2 &&&&&&&&&&&\\\hline12&&&&&&&&&&&&2 &1 &3 &&&&&&&&&&\\\hline13&&&&&&&&&&&&&3 &2 &1 &&&&&&&&&\\\hline14&&&&&&&&&&&&&&1 &3 &2 &&&&&&&&\\\hline15&&&&&&&&&&&&&&&2 &1 &3 &&&&&&&\\\hline16&&&&&&&&&&&&&&&&3 &2 &1 &&&&&&\\\hline17&&&&&&&&&&&&&&&&&1 &3 &2 &&&&&\\\hline18&&&&&&&&&&&&&&&&&&2 &1 &3 &&&&\\\hline19&&&&&&&&&&&&&&&&&&&3 &2 &1 &&&\\\hline20&&&&&&&&&&&&&&&&&&&&1 &3 &2 &&\\\hline21&&&&&&&&&&&&&&&&&&&&&2 &1 &3 &\\\hline22&&&&&&&&&&&&&&&&&&&&&&3 &2& 1\\\hline23&2&&&&&&&&&&&&&&&&&&&&&&1&3\\\hline\end{tabular}\end{adjustbox}
\caption{Total color matrix-part $1$}\label{Table b}\end{table}
\vspace{0.3cm}
\newpage
Here, the empty cells in the adjacency matrix do not receive any color in the total color matrix. Here, the least odd prime dividing $n$ is $3$, whence we give a $3$- vertex conformable coloring to the vertices. In the first row,  first entry is $1$ and  the second entry is given by $\lfloor\frac{3+1}{2}\rfloor+\frac{2-2}{2}+1=3$. This immediately gives us the entry of $n^{th}$ column to be $3-\frac{j-3}{2}$. The next row is just the first row shifted to the left, with the entry corresponding to the nonzero entries of the adjacency matrix. The cyclic pattern of the first row is $1-3-2-\ldots$. Correspondingly the cyclic pattern is adopted in the subsequent rows with first entries determined by the symmetry of the total color matrix. This gives us a total $3$ coloring of the subgraph consisting of all the vertices and one Hamiltonian cycle.\\
\vspace{0.3cm}
\pagebreak

The second part of the total color matrix is as given in the table \ref{Table c}.\\

\begin{table}[h]
\begin{adjustbox}{width=\columnwidth,center}
\begin{tabular}{|c|c|c|c|c|c|c|c|c|c|c|c|c|c|c|c|c|c|c|c|c|c|c|c|c|}\hline&0&1&2&3&4&5&6&7&8&9&10&11&12&13&14&15&16&17&18&19&20&21&22&23\\\hline
0&&&&&&4&&6&&&&8&&9&&&&7&&5&&&&\\\hline1&&&&&&&5&&7&&&&9&&8&&&&6&&4&&&\\\hline2&&&&&&&&4&&6&&&&8&&9&&&&7&&5&&\\\hline3&&&&&&&&&5&&7&&&&9&&8&&&&6&&4&\\\hline4&&&& & & &&&&4&&6&&&&8&&9&&&&7&&5 \\\hline5&4&&&& & & &&&&5&&7&&&&9&&8&&&&6&\\\hline6&&5&&&& & & &&&&4&&6&&&&8&&9&&&&7\\\hline7&6&&4&&&& & & &&&&5&&7&&&&9&&8&&&\\\hline8&&7&&5&&&& & & &&&&4&&6&&&&8&&9&&\\\hline9&&&6&&4&&&& & & &&&&5&&7&&&&9&&8&\\\hline10&&&&7&&5&&&& & & &&&&4&&6&&&&8&&9\\\hline11&8&&&&6&&4&&&& & & &&&&5&&7&&&&9&\\\hline12&&9&&&&7&&5&&&& & & &&&&4&&6&&&&8\\\hline13&9&&8&&&&6&&4&&&& & & &&&&5&&7&&&\\\hline14&&8&&9&&&&7&&5&&&& & & &&&&4&&6&&\\\hline15&&&9&&8&&&&6&&4&&&& & & &&&&5&&7&\\\hline16&&&&8&&9&&&&7&&5&&&& & & &&&&4&&6\\\hline17&7&&&&9&&8&&&&6&&4&&&& & & &&&&5&\\\hline18&&6&&&&8&&9&&&&7&&5&&&& & & &&&&4\\\hline19&5&&7&&&&9&&8&&&&6&&4&&&& & & &&&\\\hline20&&4&&6&&&&8&&9&&&&7&&5&&&& & & &&\\\hline21&&&5&&7&&&&9&&8&&&&6&&4&&&& & & &\\\hline22&&&&4&&6&&&&8&&9&&&&7&&5&&&& && \\\hline23&&&&&5&&7&&&&9&&8&&&&6&&4&&&&&\\\hline\end{tabular}\end{adjustbox}
\caption{ Total color matrix-part $2$}\label{Table c}\end{table}
\vspace{0.3cm}
\pagebreak

 Here, we have used $6$ extra colors to the remaining edges. The final total color matrix is shown in the  table \ref{Table d}.\\

\begin{table}[h]
\begin{adjustbox}{width=\columnwidth,center}
\begin{tabular}{|c|c|c|c|c|c|c|c|c|c|c|c|c|c|c|c|c|c|c|c|c|c|c|c|c|}\hline&0&1&2&3&4&5&6&7&8&9&10&11&12&13&14&15&16&17&18&19&20&21&22&23\\\hline
0&1&3&&&&4&&6&&&&8&&9&&&&7&&5&&&&2\\\hline1&3&2&1&&&&5&&7&&&&9&&8&&&&6&&4&&&\\\hline2&&1&3&2&&&&4&&6&&&&8&&9&&&&7&&5&&\\\hline3&&&2&1&3&&&&5&&7&&&&9&&8&&&&6&&4&\\\hline4&&&&3 &2 &1 &&&&4&&6&&&&8&&9&&&&7&&5 \\\hline5&4&&&&1 &3 &2 &&&&5&&7&&&&9&&8&&&&6&\\\hline6&&5&&&&2 &1 &3 &&&&4&&6&&&&8&&9&&&&7\\\hline7&6&&4&&&&3 &2 &1 &&&&5&&7&&&&9&&8&&&\\\hline8&&7&&5&&&&1 &3 &2 &&&&4&&6&&&&8&&9&&\\\hline9&&&6&&4&&&&2 &1 &3 &&&&5&&7&&&&9&&8&\\\hline10&&&&7&&5&&&&3 &2 &1 &&&&4&&6&&&&8&&9\\\hline11&8&&&&6&&4&&&& 1&3 &2 &&&&5&&7&&&&9&\\\hline12&&9&&&&7&&5&&&&2 &1 &3 &&&&4&&6&&&&8\\\hline13&9&&8&&&&6&&4&&&&3 &2 &1 &&&&5&&7&&&\\\hline14&&8&&9&&&&7&&5&&&&1 &3 &2 &&&&4&&6&&\\\hline15&&&9&&8&&&&6&&4&&&&2 &1 &3 &&&&5&&7&\\\hline16&&&&8&&9&&&&7&&5&&&&3 &2 &1 &&&&4&&6\\\hline17&7&&&&9&&8&&&&6&&4&&&&1 &3 &2 &&&&5&\\\hline18&&6&&&&8&&9&&&&7&&5&&&&2 &1 &3 &&&&4\\\hline19&5&&7&&&&9&&8&&&&6&&4&&&&3 &2 &1 &&&\\\hline20&&4&&6&&&&8&&9&&&&7&&5&&&&1 &3 &2 &&\\\hline21&&&5&&7&&&&9&&8&&&&6&&4&&&&2 &1 &3 &\\\hline22&&&&4&&6&&&&8&&9&&&&7&&5&&&&3 &2& 1\\\hline23&2&&&&5&&7&&&&9&&8&&&&6&&4&&&&1&3\\\hline\end{tabular}\end{adjustbox}
\caption{Final total color matrix}\label{Table d}\end{table}

\end{exmp}
\pagebreak
%\section{Divisible vertex coloring and Total chromatic number of some other Circulant Graphs}

 A regular graph is said to have a divisible vertex coloring if its vertices can be given at most $\Delta+1$ colors such that each color class of vertices have the same cardinality and as well as have the same parity as the number of vertices. Note that for a circulant graph $G$ on $n$ vertices with generating set $S=\{s_i\}$,  $(|S|+1)|n$ and $|S|+1 (=\Delta+1)$ not dividing any $s_i$ implies that there exists a divisible vertex coloring of $G$ with $\Delta+1$ colors, as we can classify the vertices (vertices are labled as $0, 1, ..., \Delta, . . . , n-1$) as $(0,\Delta+1,2\Delta+2,\ldots, n-\Delta-1),(1,\Delta+2,2\Delta+3,\ldots,n-\Delta),\ldots,(\Delta,2\Delta+1,\ldots,n-1)$.
\begin{thm}
Consider an odd circulant graph $G$ with generating set $S=\{s_i\}$ such that $s_i$ is not divisible by $\Delta+1$, and any two $s_i,s_k$ are not congruent modulo $\Delta+1$. Then $G$ is of type I.
\end{thm}
\begin{proof}
%The total coloring of the graph consists of labelling the diagonal elements (corresponding to coloring the vertices) and the nonzero superdiagonals (corresponding to edge coloring the graph)  of the adjacency matrix so that the colors mutually do not clash and satisfy the total coloring condition, that is, the colors in each row and column is not repeated. The matrix so labelled can be said to be the total color matrix of the graph.\\
We fill the total color matrix by following a  procedure similar to the one used in Theorem $2.2$. The entries of the first row of the total color matrix is chosen by the following rule: (here $j$ is the column number of the color matrix with $j\le\lfloor\frac{n}{2}\rfloor$)\\
i) If $j=1$ then the entry in that column is $1$.\\
ii) If $j>1$ is odd, then the entry at the $j$-th column is $(2+\frac{j-3}{2})\pmod {\Delta+1}$. Note that this would force the entry  at the $n+2-j$-th column will be $(\Delta+1-\frac{j-3}{2})\pmod {\Delta+1}$.\\
iii) If $j$ is even, then the entry at the $j$-th column is $(\frac{\Delta+2}{2}+\frac{j-2}{2}+1)\pmod {\Delta+1}$. Note that this would force the entry at the $n+2-j$-th column will be $(\frac{\Delta+2}{2}-\frac{j-2}{2})\pmod {\Delta+1}$.\\

It is seen that the entries generated by the above rules give us distinct entries and each row and/or column has only one instance of each nonzero element. This is because the additional criterion that for any two generating elements of the graph $s_i$ and $s_k$, we have  $s_i\not\equiv s_k\pmod{\Delta+1}\implies j_i\not\equiv j_k\pmod{\Delta+1}$, thereby ensuring the distinctness of the starting points of each superdiagonal. Let $a_j$ be the entry at the $j$ th position. Then the colors so put on the edges (superdiagonals) clash with the vertex colors only when $a_j\equiv j\pmod {\Delta+1}$ or $a_{n+2-j}\equiv n+2-j\pmod {\Delta+1}$ respectively. Now, using the above rules, these amount to either $2+\frac{j-3}{2}\equiv j \pmod {\Delta+1}$ and $\Delta+1-\frac{j-3}{2}\equiv n+2-j\pmod {\Delta+1}$, when $j$ is odd; or $\lfloor\frac{\Delta+2}{2}\rfloor+\frac{j-2}{2}+1\equiv j\pmod {\Delta+1}$ and $\lfloor\frac{\Delta+2}{2}\rfloor-\frac{j-2}{2}n+2-j\pmod {\Delta+1}$. Solving these congruences using the fact that $(\Delta+1)|n$, we obtain the only case $j\equiv1\pmod {\Delta+1}$, which implies $(\Delta+1)|(j-1)$. However, we note that since  $j-1$ corresponds to a generating element of the graph, and  it is additionally given that no generating element of the graph is divisible by $\Delta+1$ the situation is not possible.\\  

We then start the diagonal in the pattern $1-2-3-\ldots-\Delta+1-1-2-\ldots-$. The entries of the superdiagonals start from the entries determined by the above rules and then follow a similar pattern as in the principal diagonal.  By symmetry the subdiagonals are also filled. This fills the diagonal elements with the colors corresponding to the above divisible $\Delta+1$ vertex coloring and a $\Delta+1$ edge coloring, which satisfy the total coloring condition. This gives us a $\Delta+1$ or a type I total coloring of the graph. 
\end{proof}
\begin{exmp}
Take the Cayley graph on the group $\mathbb{Z}_{21}$ corresponding to the generating set $S=\{1,3,4,17,18,20\}$. Here, $\Delta=6$, and it satisfies all the conditions of the previous theorem with the divisible vertex color classes given by (0,7,14), (1,8,15), (2,9,16), (3,10,17), (4,11,18), (5,12,19), (6,13,20), where the vertices are the integers from $0$ to $21$ inclusive. 
\end{exmp}

 \begin{thm}
 If $G$ is a circulant graph on $n$ vertices with $n=2(2k+1)$ and $\frac{n}{2} \le \Delta < n-1$  such that $\overline{G}$ is connected. Then $G$ is type I.
 \end{thm}
 \begin{proof}
 Since $G$ is circulant, therefore $\overline{G}$ is also circulant. As $\overline{G}$ is connected, therefore $\overline{G}$  has a perfect matching $M$ as $n$ is even, by \cite{GOD}. Let $M=\{v_0v_1, v_2 v_3,\ldots, v_{n-1}v_n\}$. Thus, the pairs of vertices $\{v_0,v_1\},\{v_2,v_3\},\ldots,\{v_{n-1},v_n\}$ are independent  in $G$,  and we give a $\frac{n}{2}=2k+1$ colors to these independent sets of vertices, which is a divisible vertex coloring. This vertex coloring can be extended to a partial total coloring of $G$, say the total coloring of a sub graph $H$, where  $H$ is  so chosen that $G-E(H)$  is connected (which can always be done as the graph has degree $\ge\frac{n}{2}$). 
 \\
 
%This implies that the generating set of $G-E(H)$ actually generates the group $\mathbb{Z}_n$ (this can be always done as degree $\ge\frac{n}{2}$).
 This  total coloring of $H$ (the partial total coloring of $G$) can be done by a procedure similar to that in the previous Theorem 2.3. For that the $2k+1$  entries corresponding to the first row of $H$ are chosen by the following rules (here $j$ is the column number of the total color matrix, $j\le\frac{n}{2}$): \\
 i) If $j=1$ then the entry in that column is $1$.\\
 ii) If $j$ is odd, then the entry at the $j$-th column is $(2+\frac{j-3}{2})\mod ({2k+1})$. \\
 Note that this would force the entry  at the $n+2-j$-th superdiagonal to be $(2k+1-\frac{j-3}{2})\mod ({2k+1})$.\\
 iii) If $j$ is even, then the entry at the $j$-th column is $(\frac{2k+2}{2}+\frac{j-2}{2}+1)\mod ({2k+1})$. \\
 Note that this would force the entry at the $n+2-j$-th column to be $(\frac{2k+2}{2}-\frac{j-2}{2})\mod ({2k+1})$.\\

 In this case, as in the previous theorem, the distinctness of the starting entries of the subdiagonals are ensured as $j\le\frac{n}{2}$ and no two column numbers are congruent modulo $2k+1$.
 By using the above entries as a starting entry, we could fill the chosen $2k$ superdiagonals as in the previous theorem. The entries of the superdiagonals (corresponding to the colors of the edges) clash with the vertex colors only when $a_j\equiv j\mod ({2k+1})$ or $a_{n+2-j}\equiv (n+2-j)\mod ({2k+1})$ respectively. Now, using the above rules, these amount to either $(2+\frac{j-3}{2})\equiv j \mod ({2k+1})$ and $(2k+1-\frac{j-3}{2})\equiv (n+2-j)\mod ({2k+1})$, when $j$ is odd; or $(\frac{2k+2}{2}+\frac{j-2}{2}+1)\equiv j\mod ({2k+1})$ and $(\frac{2k+2}{2}-\frac{j-2}{2})\equiv (n+2-j)\mod ({2k+1})$. Solving these congruences using the fact that $(2k+1)|n$, we obtain the only case $j\equiv 1\mod ({2k+1})$, which implies $(2k+1)|(j-1)$. However, we note that since $2k+1=\frac{n}{2}$ and $j$ ranges from $1$ to $\frac{n}{2}$, such a condition does not occur. Hence, the non-clashing of the edge colors with the vertex colors is thereby satisfied.\\

 The remaining $\frac{\Delta}{2}-k$ non-zero super (sub)  diagonals  can be easily filled  with $\Delta-2k$ colors, as it corresponds to an edge coloring of the graph $G-E(H)$, which is of class I by \cite{STO}. Thus, we used  $(\Delta-2k)+(2k+1)=\Delta+1$, to color $G$. Hence $G$  is type I.
 \end{proof}
Since  we can fill the entries of the color matrix of the graphs in the above four theorems in a linear number of steps compared to the input, the algorithm used to give the total coloring is polynomial time in all the above cases.
\begin{thm}
If $G$ be a non-complete Cayley graph of odd order $n$ and clique number $\omega>\frac{n}{3}$. Then, $G$ is not of type I.
\end{thm}
\begin{proof}
It is easy to see that a Cayley graph with clique size $\omega>\frac{n}{3}$ will not have an isolated maximal clique if it is not complete.   On the other hand, the graph would have exactly $n$ number of maximal cliques. This is seen as follows: Let $\{g_i\}$ be the elements of the group underlying $G$ and $\{s_1,. . . , s_k\}$ the elements of the generating set of $G$. Suppose $g_1-g_1s_1-g_1s_1s_2-\ldots-g_1s_1s_2\cdots s_m-g_1, \  m \leq k,$  be the sequence of vertices of a clique. Then, $g_2-g_2s_1-g_2s_1s_2-\ldots-g_2s_1s_2\cdots s_m-g_2, \  m \leq k,$ are also the vertices of a clique for any $g_1,g_2$. Since this sequence of vertices can be written for any element $g_i$ and there are $n$ elements in the group, hence we have $n$ different maximal cliques. Note that the cliques can intersect and may share some edges as one or more vertices in the sequences can be the same.\\

Now, for a type I total coloring, the graph should have a $\Delta+1$ conformable coloring, which implies, by the regularity of $G$, that each color class should have the same parity as that of $n$, which is odd. Since the clique size $\omega >\frac{n}{3}$, we cannot put more than two  vertices in any color class (for otherwise, there might be two adjacent vertices in the same color class). The only way the graph can have a conformable coloring is that each color class have exactly one vertex, but such a coloring would exceed $\Delta+1$ colors as the graph is not complete. Thus, the graph does not have a $\Delta+1$ conformable coloring, whence the graph is not type I.  
\end{proof}
\begin{thm}
Let $G$ be is a perfect Cayley graph with $\chi(G)|n$ and $\chi(G)$ being odd. Then $G$ satisfies TCC.
\end{thm}
\begin{proof}
Since $G$ is perfect, therefore the clique number is also equal to the chromatic number, which is $\chi$, an odd integer. Now, the case when the graph is complete is trivial. When the graph is not complete, as in the previous theorem, the graph will not have an isolated clique, rather has $\frac{n}{\chi}$ disjoint maximal cliques. The reason for this is as follows: Since we have $n$ vertices and the clique number is $\omega$, we have $n$ different $\omega$-cliques (of course several intersecting). Hence, any maximal independent set would contain at most $\frac{n}{\omega}$ vertices in it.   By \cite{VER}, we have that every maximal clique intersects every maximal independent set. Therefore each disjoint maximal independent set of vertices has one vertex of each of a maximal clique. Thus, there are $\frac{n}{\omega}$ vertex disjoint maximal cliques. In fact, the vertices would be distributed equitably in $\omega$ independent sets in a $\omega$ coloring.\\

Now, we proceed to coloring. We first  color the vertices of $G$ with $\chi$ colors. Using the same $\chi$ colors, we can color the edges of the disjoint cliques such that each of the disjoint cliques are colored totally. Now, let the union of $\frac{n}{\chi}$ cliques be the graph $G'$. Consider the graph $G-E(G')$. We observe that $G-E(G')$ is a $\Delta-\chi+1$-regular graph, hence requires at most $\Delta-\chi+2$ extra colors to color the edges. Thus, the total number of colors required in a total coloring of $G$ is at most  $\chi+(\Delta-\chi+2)=\Delta+2$ colors, which shows that $G$ satisfies TCC.
\end{proof}
Note that if $G-E(G')$ be a class I graph then the $G$ would be a type I graph.

\end{document}